\documentclass[12pt, reqno, a4paper]{amsart}

\usepackage{ amssymb, amsmath, enumerate, amsfonts, amsthm, mathrsfs, url, bm, mathtools}

\usepackage{xcolor}  	
\usepackage{hyperref}
\hypersetup{
colorlinks,
linkcolor={cyan!80!black},
citecolor={cyan!80!black},
urlcolor={cyan!80!black}
}

\usepackage{color}

\usepackage[margin=1in]{geometry}

\RequirePackage{doi}

\usepackage{upgreek}
\usepackage[OT1]{fontenc}
\usepackage{amscd}
\usepackage{amsfonts}
\usepackage{float}
\usepackage{color}
\usepackage[
backend=biber,
style=alphabetic,
]{biblatex}
\usepackage{bookmark}

\renewbibmacro{in:}{}
\DeclareFieldFormat{title}{#1}

\DeclareFieldFormat[article]{title}{\mkbibemph{#1}}       % Articles
\DeclareFieldFormat[incollection]{title}{\mkbibemph{#1}}  % Incollection chapters
\DeclareFieldFormat[book]{title}{\mkbibemph{#1}}          % Standalone books
\DeclareFieldFormat[incollection]{booktitle}{#1}          % Plain for incollection book titles
\DeclareFieldFormat[article]{journaltitle}{#1}            % Plain for journals

\AtEveryBibitem{%
  \ifentrytype{misc}{\DeclareFieldFormat{title}{\mkbibemph{#1}}}{}}

\DeclareFieldFormat{eprint:eprint}{arXiv:\href{https://arxiv.org/abs/#1}{#1}}

\DeclareFieldFormat[inproceedings]{title}{\mkbibemph{#1}}

\DeclareFieldFormat[inproceedings]{booktitle}{#1}

\usepackage{amssymb}
\newtheorem{theorem}{Theorem}[section]
\newtheorem{lemma}{Lemma}[section]

\newtheorem{proposition}{Proposition}[section]

\theoremstyle{definition}
\newtheorem{definition}{Definition}[section]

\theoremstyle{remark}

\numberwithin{equation}{section}

\newcommand{\Mod}[1]{\ (\mathrm{mod}\ #1)}

\renewcommand{\leq}{\leqslant}
\renewcommand{\geq}{\geqslant}
\newcommand{\eps}{\varepsilon}
\newcommand{\N}{\mathbb N}
\newcommand{\R}{\mathbb R}
\newcommand{\Z}{\mathbb Z}
\newcommand{\ind}{\mathbf 1}
\newcommand{\ee}{\mathrm e}
\newcommand{\calE}{\mathcal E}

\newcommand{\Bohr}[1]{\ind_{\|\alpha #1+\beta\|\leq\delta}}

\title[Diophantine approximation by primes and Landau--Siegel zeros]
{Diophantine approximation by primes and Landau--Siegel zeros}
\author{Sun-Kai Leung}
\address{Mathematical Institute, University of Oxford, Andrew Wiles Building,
Radcliffe Observatory Quarter, Woodstock Rd, Oxford OX2 6GG, United Kingdom}
\email{sunkaileung@gmail.com}
\author{Stelios Sachpazis}
\address{Charles University, Faculty of Mathematics and Physics,
Department of Algebra, Sokolovsk\'{a} 49/83, 186 75 Praha 8, Czech Republic}
\email{stylianos.sachpazis@matfyz.cuni.cz}
\subjclass[2020]{Primary 11J20; Secondary 11N05, 11N36, 11M20}
\date{}

\begin{document}
\begin{abstract}
Let $\alpha>0$ be an irrational number of finite type and let $\beta\in\R$.
Assuming the infinitude of Siegel zeros (respectively, sufficiently strong Siegel zeros), we show that for every sufficiently small $\eps>0$, there
exist infinitely many primes $p$ such that
\[
\|\alpha p+\beta\|\leq p^{-1/3+\eps}
\qquad\text{(respectively, }\|\alpha p+\beta\|\leq p^{-1/3-1/20}\text{)}.
\]
\end{abstract}
\maketitle

\section{Introduction}

Let $\alpha>0$ be irrational and let $\beta\in\R$. A classical Diophantine approximation problem
asks for which $\gamma>0$ the inequality
\begin{equation*}
\|\alpha p+\beta\|\leq p^{-\gamma+\eps}
\end{equation*}
has infinitely many prime solutions $p$ for every $\eps>0$.
Vinogradov \cite{MR2104806} obtained $\gamma=1/5$. This was successively improved to $1/4$ by Vaughan \cite{MR472731}, to $3/10$ by Harman
\cite{MR686496}, to $4/13$ by Jia \cite{MR1247382}, to $7/22$ by Harman
\cite{MR1367078}, and to $9/28$ by Jia \cite{MR1790174}.
In the homogeneous case $\beta=0$, Heath-Brown and Jia
\cite{MR1863396} obtained $16/49$, and Matom\"aki \cite{MR2525926}
subsequently reached $1/3$.\footnote{James Maynard obtained some fixed $\gamma>1/3$ in some of his unpublished notes, and we thank him for generously sharing his results with us.}

The possible existence of Landau--Siegel zeros provides another line of attack to problems about primes. Although the existence of these exceptional zeros is considered unlikely, it often leads to intriguing implications in the study of prime numbers. For example, one can appeal to a result of Heath-Brown \cite{MR0703977} from 1983 to show that the infinitude of Siegel zeros implies the twin prime conjecture. %Later, Friedlander and Iwaniec assumed the existence of Landau-Siegel zeros and developed a conditional approach for primes in arithmetic progressions and short intervals \cite{FI03,FI04}, as well as for more general sequences in their illusory sieve \cite{FI05}. 

For convenience, we now define the \textit{strength} of Siegel zeros.

\begin{definition}[Siegel zeros of strength $k$]
Fix an integer $k\geq1$. Let $\chi_1,\chi_2,\ldots$ be primitive real
Dirichlet characters with strictly increasing conductors
$D_1<D_2<\cdots$. Suppose that $L(s,\chi_j)$ has a real zero
$\beta_j\in(0,1)$ for each $j$, and that
\[
1-\beta_j=o((\log D_j)^{-k})\qquad \text{as $j\to\infty$}.
\]
We call $(\beta_j)$ a sequence of \emph{Siegel zeros of strength $k$}.
When $k=1$, we simply call it a sequence of \emph{Siegel zeros}.
\end{definition}

To state our results, let us also recall that an irrational number $\alpha$ is said to be of \emph{finite type} if
\[
\uptau(\alpha):=\sup\{t>0:\liminf_{n\to\infty}n^t\|n\alpha\|=0\}<\infty.
\]

Our first result conditionally extends the one-third exponent of
Matom\"aki to the inhomogeneous setting for finite-type irrationals.

\begin{theorem}\label{thm:main}
Assume the existence of infinitely many Siegel zeros. Let $\alpha>0$
be an irrational number of finite type and let $\beta\in\R$.
Then, for every $\eps>0$, there exist infinitely many primes $p$ such that
\[
\|\alpha p+\beta\|\leq p^{-1/3+\eps}.
\]
\end{theorem}

Under a stronger hypothesis on the Siegel zeros, we can go beyond one-third.

\begin{theorem}\label{thm:harman}
Assume the existence of infinitely many Siegel zeros of strength $7$.
Let $\alpha>0$ be an irrational number of finite type and let $\beta\in\R$.
Then, for every $0<\eta<\eta_0$, there exist infinitely many primes $p$
such that
\[
\|\alpha p+\beta\|\leq p^{-1/3-\eta},
\]
where
\[
\eta_0:=\frac{\sqrt{\ee}-1}{3(1+2\sqrt{\ee})}
=0.0503183978\ldots.
\]
\end{theorem}

To prove our main theorems, we first establish estimates at scales for which $\alpha$ admits a suitable rational approximation. The finite-type assumption enters only in
Section~\ref{sec:mainproofs}, where a synchronization lemma
supplies such scales for both theorems.

\medskip
\noindent\textit{Notation.}
Throughout the paper, we use the standard big $O$ and little $o$
notation, as well as Vinogradov's notation $\ll$ and $\gg,$ where the implied constants depend only on the subscripted parameters, unless specified otherwise.
For $x \in \mathbb{R},$ we set $e(x):=e^{2\pi i x}$, and let
 $\|x\|$ denote the distance from $x$ to the nearest integer. For $z \geq 1,$ we let
$P(z):=\prod_{p\leq z}p$, and for a positive integer $n,$ we let $P^-(n)$ denote its least prime factor (with the convention that $P^-(1)=\infty$). We also write $n\sim y$ for
$y<n\leq2y$. A sequence $(c_d)$ of complex numbers is called divisor-bounded if there exists an absolute constant $C>0$ such that $|c_d|\ll \tau(d)^C$ for all $d \in \N$, where $\tau$ is the divisor function. Finally, we denote the Euler--Mascheroni constant by $\gamma_*$.

\section{Landau--Siegel zeros: from primes to divisor-like functions}

Let $\chi$ be a primitive real Dirichlet character of conductor $D\geq2$,
and suppose that the Dirichlet series $L(\cdot,\chi)$ of the character has a real zero
\[
\beta_\chi=1-\frac1{\xi\log D}\quad\text{with}\quad \xi\geq10.
\]
We refer to such an individual exceptional zero as a
\emph{Landau--Siegel zero}. The associated bias towards
$\chi(p)=-1$ makes $\chi$ resemble the M\"obius function $\mu$ on large
primes. This suggests replacing $\Lambda=\mu*\log$ by the divisor-like
function $\chi*\log$, after removing small prime factors. The
rough-number estimate in Lemma~\ref{lem:roughlambda} makes this
observation quantitative. Denote
\[
\lambda:=\ind*\chi,\qquad \lambda':=\chi*\log.
\]
Since $\lambda(p^a)=1+\chi(p)+\cdots+\chi(p)^a\geq0$, the function
$\lambda$ is non-negative. Moreover,
\[
\lambda'=\chi*(\ind*\Lambda)=(\ind*\chi)*\Lambda=\lambda*\Lambda,
\]
so $\lambda'$ is non-negative as well, and
\begin{equation}\label{lL}
\Lambda(n)=\lambda'(n)-\sum_{\substack{k\ell=n\\k>1}}
\lambda(k)\Lambda(\ell).
\end{equation}
This identity is the starting point of the exceptional-character sieve (see, for instance, \cite[equation (10)]{MM23}). In particular, thanks to (\ref{lL}), the passage
from primes to divisor-like functions is exact, and any contributions arising from the convolution term will be shown to be negligible.

\section{Asymptotic estimates: at one-third}

To prove Theorem~\ref{thm:main}, we first establish an asymptotic
formula at scales for which $\alpha$ admits a suitable rational
approximation. Our starting point is \eqref{lL}.

\begin{theorem}\label{thm:asymp}
Fix $\eps\in(0,1/12]$ and let $\chi$ be a primitive quadratic character of conductor $D\geq2$
such that its Dirichlet series $L(\cdot,\chi)$ vanishes at the real number
\[
\beta_\chi=1-\frac1{\xi\log D}\quad\text{for some}\quad\xi\geq10.
\]
Set $x=D^V$ with $V\geq20000/\eps$, and $\delta=x^{-1/3+\eps}$.
Let also $\alpha>0$ be irrational and
$\beta\in\R$, and
suppose that there exist coprime positive integers $a$ and $q$ satisfying the following inequalities:
\begin{equation}\label{eq:A-diophantine-hypothesis}
\left|\alpha-\frac aq\right|\leq \frac{1}{q^2},\qquad
x^{1/3}\leq q\leq x^{2/3}.
\end{equation}
There exists a constant $c_\eps>0$, depending only on $\eps$, such that we have, uniformly in $\beta$,
\begin{equation*}
\sum_{\substack{n\leq x\\\|\alpha n+\beta\|\leq\delta}}\Lambda(n)
=2\delta\psi(x)\Bigg\{1+O_\eps\Bigg(\frac{V^{16}}{\xi}
+\exp\big(-c_\eps\sqrt{V\log\xi}\big)\Bigg)\Bigg\},
\end{equation*}
where $\psi(x)=\sum_{n\leq x}\Lambda(n)$.
\end{theorem}

Throughout this section and Sections~\ref{sec:s1}--\ref{sec:delta},
the hypotheses of Theorem~\ref{thm:asymp} are in force. We use the following parameters throughout these sections:
\begin{equation*}
v:=\min\left\{\sqrt{\frac V{\log\xi}},\,2\right\},\qquad
z:=D^v,\qquad u:=\frac{\eps V}{100v},\qquad
R:=z^u=x^{\eps/100},
\end{equation*}
Henceforth, for the sake of notational convenience, we write
\begin{equation*}
\calE:=V^{16}\xi^{-1}
+\exp(-c_\eps\sqrt{V\log\xi}).
\end{equation*}

We prove Theorem \ref{thm:asymp} by restricting identity \eqref{lL} to $z$-rough integers. This pre-sieving move reduces the theorem's proof to the following three estimates.

\begin{proposition}\label{prop:s1}
With the above notation, let
\begin{equation*}
S_1:=\sum_{\substack{k\ell\leq x\\k,\ell>1\\P^-(k\ell)>z}}
\lambda(k)\Lambda(\ell).
\end{equation*}
Then $S_1\ll_\eps \psi(x)\calE$.
\end{proposition}

\begin{proposition}\label{prop:s2}
With the above notation, let
\begin{equation*}
S_2:=\sum_{\substack{k\ell\leq x\\k,\ell>1\\P^-(k\ell)>z\\
\|\alpha k\ell+\beta\|\leq\delta}}\lambda(k)\Lambda(\ell).
\end{equation*}
Then $S_2\ll_\eps \delta\psi(x)\calE$.
\end{proposition}

\begin{proposition}\label{prop:delta}
With the above notation, let
\begin{equation*}
\Delta:=\sum_{\substack{n\leq x\\P^-(n)>z\\
\|\alpha n+\beta\|\leq\delta}}\lambda'(n)
-2\delta\sum_{\substack{n\leq x\\P^-(n)>z}}\lambda'(n).
\end{equation*}
Then $|\Delta|\ll_\eps \delta\psi(x)\calE$.
\end{proposition}

\begin{proof}[Proof of Theorem~\ref{thm:asymp}]
The contribution of the prime powers whose prime base is at most $z$ is readily bounded by
\begin{equation*}
\sum_{\substack{n\leq x\\P^-(n)\leq z}}\Lambda(n)
\leq\sum_{p\leq z}\log p\left\lfloor\frac{\log x}{\log p}\right\rfloor
\ll\frac{z\log x}{\log z}.
\end{equation*}
Summing \eqref{lL} over $z$-rough integers, first with the Bohr condition
and then without it, gives
\[
\sum_{\substack{n\leq x\\\|\alpha n+\beta\|\leq\delta}}\Lambda(n)
-2\delta\psi(x)
=\Delta-S_2+2\delta S_1+O\left(\frac{z\log x}{\log z}\right).
\]
Since $\eps\in(0,1/12]$, we have $V>6$, so
$z\leq D^2<x^{1/3}<x^{2/3}\delta\ll x\delta/\xi$.
Here $\xi\ll D^2$ follows from Siegel's lower bound for $L(1,\chi)$
(see \cite[Corollary 11.15]{montgomery2007multiplicative}).
Thus the error above is $\ll x\delta V/(v\xi)\ll x\delta\calE$,
and the theorem follows from the three propositions.
\end{proof}

\section{Lower bounds: beyond one-third}

To prove Theorem \ref{thm:harman}, we first establish a lower bound at scales for which $\alpha$ admits a suitable rational approximation using the general
exceptional-character sieve of Merikoski \cite[Theorem 16]{Mer24}.

\begin{theorem}\label{thm:lower}
Let $0<\eta<\gamma<\eta_0$, where $\eta_0$ is the constant defined in Theorem \ref{thm:harman}. Let also $\alpha>0$ be irrational and $\beta\in\R$, and suppose that there exist a real number $\kappa\in(\eta,1/6)$ and coprime positive integers $a$ and $q$ such that
\begin{equation}\label{eq:lowerq}
\left|\alpha-\frac aq\right|\leq \frac{1}{q^2}\quad\text{and}\quad
x^{1/3+\kappa}\leq q\leq x^{1-\kappa}.
\end{equation}
Assume further that $\chi$ is a primitive real Dirichlet character of conductor $D\geq 2$, and let $V=V(D)>1$ be a parameter depending on $D$ such that
\[
V\to\infty\quad\text{and}\quad\log V=o(\log D)\quad\text{as }D\to\infty.
\]
If $x=D^V$ and $\delta=x^{-1/3-\eta}$, then, uniformly in $\beta$, we have that
\begin{equation*}
\sum_{\substack{n\leq x\\\|\alpha n+\beta\|\leq\delta}}\Lambda(n)
\geq2\delta\left\{c(\gamma)-O(L(1,\chi)(\log x)^5)-o(1)\right\}\psi(x),
\end{equation*}
where $\psi(x)=\sum_{n\leq x}\Lambda(n)$ and
\begin{equation*}
c(\gamma):=1-2\log\left(\frac{1+3\gamma}{1-6\gamma}\right).
\end{equation*}
The implied constants may depend on $\eta,\gamma$ and $\kappa$.
\end{theorem}

In this section and Sections~\ref{sec:typei}--\ref{sec:sievehypothesis},
we use the hypotheses and notation of Theorem~\ref{thm:lower}. Now, for $y\in[x/\log x,x/2]$, define the sequences $(a_n)$ and $(b_n)$ given by
\begin{equation*}
a_n:=\ind_{y<n\leq2y}\ind_{(n,D)=1}\Bohr{n}\quad\text{and}\quad
b_n:=2\delta\ind_{y<n\leq2y}\ind_{(n,D)=1}\quad(n\in\N).
\end{equation*}
We also define the non-negative multiplicative function $g$ and the quantity $X$ by
\begin{align}
g(d)&:=\frac{\ind_{(d,D)=1}}d\quad(d\in\N),\label{eq:g}\\
X&:=\sum_n b_n
=2\delta y \cdot \frac{\varphi(D)}D+O(\delta\tau(D)).\label{eq:X}
\end{align}

To apply \cite[Theorem~16]{Mer24}, we need ordinary and
$\chi$-twisted Type I estimates, the corresponding bounds for the
reference sequence (also on short intervals), a dimension-one density,
the prime normalization, and the exceptional-character bounds.
Proposition~\ref{prop:typei} gives the two Type I comparisons; the
remaining conditions are verified in Proposition~\ref{prop:sievehypothesis}
and its proof.

\begin{proposition} \label{prop:typei}
Let $x$ be as in Theorem \ref{thm:lower} and $y\in[x/\log x,x/2]$. Let also $(c_d)$ be a divisor-bounded sequence of complex numbers, $A>0$, and $M\leq y^{2/3-\gamma}$. For arbitrary intervals $J_d\subseteq(y/d,2y/d]$, we have that
\begin{gather}
\left|\sum_{d\leq M}c_d\sum_{m\in J_d}(a_{dm}-b_{dm})\right|
\ll_AX(\log y)^{-A},\notag\\
\left|\sum_{d\leq M}c_d\sum_{m\in J_d}\chi(m)(a_{dm}-b_{dm})\right|
\ll_AX(\log y)^{-A}.\label{eq:typeitwist}
\end{gather}
Both estimates are uniform in $y$, $\beta$ and the intervals $J_d$.
The implied constants may depend on the divisor bound for $(c_d)$,
as well as on $\eta,\gamma$ and $\kappa$.
\end{proposition}

\begin{proposition}\label{prop:sievehypothesis}
Uniformly for $x/\log x\leq y\leq x/2$, the sequence $(b_n)$ satisfies
\begin{gather}
\sum_n b_n\Lambda(n)
=(1+o(1))\ee^{\gamma_*}\log w
\prod_{p\leq w}(1-g(p))X
\sim2\delta y,\label{eq:normalization}\\
\sum_{k\sim w}\Lambda(k)g(k)
=(1+o(1))\sum_{k\sim w}\frac{\Lambda(k)}k
\label{eq:referenceg}
\end{gather}
for $w>y^a$, where $a>0$ is fixed. Moreover, for $t>w\geq D^9$,
\begin{equation}\label{eq:exceptionalaxiom}
\left|\sum_{n\leq t}\chi(n)g(n)\right|\ll L(1,\chi),\qquad
\sum_{w<n\leq t}\lambda(n)g(n)\ll L(1,\chi)\log t.
\end{equation}
\end{proposition}

\begin{proof}[Proof of Theorem~\ref{thm:lower}]
For any $y$ in $[x/\log x,x/2]$, the scale assumptions $V\to\infty$ and $\log V=o(\log D)$ as $D\to\infty$ imply that
\begin{equation*}
D=y^{o(1)}\quad\text{and}\quad\log\log y=o(\log D)\quad\text{as }D\to\infty.
\end{equation*}
Hence, $D\gg_A(\log y)^A$ for every fixed $A>0$.
By Proposition~\ref{prop:typei} and the verifications in the proof
of Proposition~\ref{prop:sievehypothesis}, \cite[Theorem~16]{Mer24}
applies with exponent of distribution $2/3-\gamma$, under the
normalization \eqref{eq:normalization}.\footnote{In the first normalization display
of \cite[Theorem~16]{Mer24}, the factor $1/(\ee^{\gamma_*}\log w)$
must read $\ee^{\gamma_*}\log w$. We use the normalization in
\cite[Lemma~2]{Mer24}, which is also the one used in
the sieve argument.}
Therefore, if we let $J=\lfloor\log_2\log x\rfloor$, then for each $j\in\{0,1,\ldots,J-1\}$, we obtain
\begin{equation}\label{eq:dyadicsieve}
\sum_{\substack{n\sim x/2^{j+1}\\(n,D)=1,\,\|\alpha n+\beta\|\leq\delta}}\Lambda(n)
\geq2\delta\left\{c(\gamma)-O(L(1,\chi)(\log x)^5)-o(1)\right\}
\sum_{\substack{n\sim x/2^{j+1}\\(n,D)=1}}\Lambda(n).
\end{equation}

Taking the errors in the coefficient in braces in \eqref{eq:dyadicsieve}
to be nonnegative, the conclusion is immediate if that coefficient is
nonpositive. Otherwise it is bounded above by $c(\gamma)$.
Summing \eqref{eq:dyadicsieve} for
$j\in\{0,1,\ldots,J-1\}$, we infer that
\begin{align*}
\sum_{\substack{n\leq x\\\|\alpha n+\beta\|\leq\delta}}\Lambda(n)
&\geq\sum_{\substack{x/2^J<n\leq x\\\|\alpha n+\beta\|\leq\delta}}\Lambda(n)\geq\sum_{\substack{x/2^J<n\leq x\\(n,D)=1,\,\|\alpha n+\beta\|\leq\delta}}\Lambda(n)\\
&\geq 2\delta\left\{c(\gamma)-O(L(1,\chi)(\log x)^5)-o(1)\right\}\sum_{\substack{x/2^J<n\leq x\\(n,D)=1}}\Lambda(n).
\end{align*}

Since $D=x^{o(1)}$, we have
$\sum_{n\leq x,\,(n,D)>1}\Lambda(n)\ll\omega(D)\log x=x^{o(1)}$.
Also, Chebyshev's estimates give
$\sum_{n\leq x/2^{J}}\Lambda(n)\ll x/2^J\ll x/\log x$.
The prime number theorem completes the proof.
\end{proof}

\section{Proofs of Theorems \texorpdfstring{\ref{thm:main} and
\ref{thm:harman}}{1.1 and 1.2}}\label{sec:mainproofs}

Theorems~\ref{thm:asymp} and~\ref{thm:lower} are local in the parameters in the sense that they require a scale $x=D^V$ at which a convergent denominator of $\alpha$
lies in a prescribed range. The following lemma supplies such scales for
every fixed strength $k$. This is the only place where the finite-type
hypothesis on the irrational number $\alpha$ is used.

\begin{lemma}[Diophantine synchronization]\label{lem:synchronization}
Let $\alpha>0$ be an irrational number of finite type, $\vartheta\in(0,1)$, and $(\beta_j)$ be a sequence of Siegel zeros of fixed strength $k\geq1$, associated with characters of conductors $D_j$. There exist pairs of coprime positive integers $a_j$ and $q_j$, as well as real numbers $V_j$, defined as $V_j\vcentcolon=2(1+\vartheta)^{-1}\log q_j/(\log D_j)$, such that
\begin{gather*}
\left|\alpha-\frac{a_j}{q_j}\right|\leq \frac{1}{q_j^2}\quad\text{for all }j\in\N,\quad\text{and}\\
\quad V_j\to\infty,\quad \log V_j=o(\log D_j),\quad
V_j^{16}\xi_j^{-1}\to0
\quad \text{as $j \to \infty,$}
\end{gather*}
where $\xi_j:=1/((1-\beta_j)\log D_j)$ for all $j\in\N$. If $k\geq2$, we further have that
\begin{equation}\label{eq:syncstrength}
(1-\beta_j)(\log D_j)^2(\log q_j)^{k-2}\to0 \qquad \text{as $j \to \infty.$}
\end{equation}
\end{lemma}

\begin{proof}
Choose $t_\alpha>\max\{1,\uptau(\alpha)\}$, so that
$\|m\alpha\|\geq c_\alpha m^{-t_\alpha}$ for every $m\geq1$ and some
$c_\alpha>0$. If $A_r/Q_r$ are the convergents of $\alpha$, then
$\|Q_r\alpha\|\leq|Q_r\alpha-A_r|<Q_{r+1}^{-1}$. Hence
\begin{equation}\label{eq:convergentgrowth}
Q_{r+1}\leq c_\alpha^{-1}Q_r^{t_\alpha}\quad(r\in\N).
\end{equation}
Put
\[
s_j:=\frac1{(1-\beta_j)(\log D_j)^k}\quad\text{and}\quad
T_j:=\min\{s_j^{1/(100k)},\log D_j\}.
\]
Then $s_j,T_j\to\infty$. Choose $q_j$ to be the first convergent
denominator exceeding $D_j^{T_j}$, with corresponding numerator $a_j$.
The convergent estimate gives $|\alpha-a_j/q_j|<q_j^{-2}$, and
\eqref{eq:convergentgrowth} gives
$D_j^{T_j}<q_j\ll_\alpha D_j^{t_\alpha T_j}$.
Thus $V_j\asymp_{\alpha,\vartheta}T_j\to\infty$ and
$\log V_j=O_{\alpha,\vartheta}(1+\log\log D_j)=o(\log D_j)$.
Since $\xi_j=s_j(\log D_j)^{k-1}$, we also have
\[
V_j^{16}\xi_j^{-1}\ll_{\alpha,\vartheta}T_j^{16}s_j^{-1}
\leq s_j^{-1+16/(100k)}\to0.
\]
For $k\geq2$, the expression in \eqref{eq:syncstrength} is
$\ll_\vartheta V_j^{k-2}/s_j\ll_{\alpha,\vartheta}s_j^{-99/100}\to0$.
This completes the proof.
\end{proof}

\begin{proof}[Proof of Theorem~\ref{thm:main}]
Let $D_j$ be the conductors associated with the assumed Siegel zeros
$(\beta_j)$. It suffices to consider $\eps\in(0,1/12]$.
Apply Lemma~\ref{lem:synchronization} with $k=1$ and
$\vartheta=1/3-\eps$, and put $x_j=D_j^{V_j}$.
Then $q_j=x_j^{2/3-\eps/2}$ satisfies
\eqref{eq:A-diophantine-hypothesis}. Theorem~\ref{thm:asymp} and the
prime number theorem give
\[
\sum_{\substack{n\leq x_j\\\|\alpha n+\beta\|\leq x_j^{-1/3+\eps}}}
\Lambda(n)\sim2x_j^{2/3+\eps} \qquad \text{as $j \to \infty.$}
\]
The proper prime powers contribute $O(x_j^{1/2}(\log x_j)^2)$.
Hence the number of primes in the sum tends to infinity. Each prime $p$ in that sum satisfies $\|\alpha p+\beta\|\leq x_j^{-1/3+\eps}\leq p^{-1/3+\eps}$, and the theorem follows.
\end{proof}

\begin{proof}[Proof of Theorem~\ref{thm:harman}]
Fix $0<\eta<\gamma<\eta_0$ and apply
Lemma~\ref{lem:synchronization} with $k=7$ and
$\vartheta=1/3+\eta$. Let $D_j$ be the conductor associated with
$\beta_j$, and put $x_j\vcentcolon=D_j^{V_j}=q_j^{2/(1+\vartheta)}$.
Then $q_j=x_j^{2/3+\eta/2}$ satisfies \eqref{eq:lowerq} for every
$\kappa\in(\eta,1/6)$. The remaining conditions of
Theorem~\ref{thm:lower} hold with $a=a_j$, $q=q_j$, $x=x_j$ and
$V=V_j$, and so the theorem yields
\begin{align}\label{1}
\sum_{\substack{n\leq x_j\\\|\alpha n+\beta\|\leq x_j^{-\vartheta}}}\Lambda(n)
\geq\left\{c(\gamma)-O(L(1,\chi_j)(\log x_j)^5)-o(1)\right\}
2x_j^{-\vartheta}\psi(x_j).
\end{align}

By \cite[Theorem~11.4]{montgomery2007multiplicative},
$L(1,\chi_j)\ll(1-\beta_j)(\log D_j)^2$.
Together with \eqref{eq:syncstrength} for $k=7$ and
$\log x_j\asymp_\vartheta\log q_j$, this gives
$L(1,\chi_j)(\log x_j)^5=o(1)$.
Since $c(\gamma)>0$ for $\gamma<\eta_0$, the sum in \eqref{1} is
$\gg_{\eta,\gamma}x_j^{1-\vartheta}$.
The proper prime powers contribute $O(x_j^{1/2}(\log x_j)^2)$,
which is negligible since $\vartheta<1/2$.
Hence the number of primes in the sum tends to infinity, and each
satisfies $\|\alpha p+\beta\|\leq x_j^{-\vartheta}\leq p^{-1/3-\eta}$.
This completes the proof.
\end{proof}

\section{Preliminaries}

In this section, we collect the lemmata required for the proofs of the propositions.

\begin{lemma}[Fourier majorant and minorant]\label{lem:fourier}
Let $t$ and $\delta$ be such that $0<t\leq\delta/2$ and $\delta+t<1/2$. There exist smooth, one-periodic functions $\Phi^\pm$ such that
\[
0\leq\Phi^-\leq\ind_{\|\cdot\|\leq\delta}\leq\Phi^+\leq1,
\qquad \widehat\Phi^\pm(0)=2\delta+O(t).
\]
For $h\ne0$ and every fixed $A>0$,
\begin{equation}\label{bPhi}
|\widehat\Phi^\pm(h)|\ll_A
\min\{\delta,|h|^{-1}\}(1+|h|t)^{-A}.    
\end{equation}
Consequently, if $Ht\geq1$, then uniformly in $s$,
\begin{equation}\label{eq:fourierexpansion}
\Phi^\pm(s)=2\delta
+\sum_{0<|h|\leq H}\widehat\Phi^\pm(h)e(hs)
+O(t)+O_A((Ht)^{-A}).
\end{equation}
\end{lemma}

\begin{proof}
This lemma is in the spirit of
\cite[Lemma~2.1]{Har07}. For completeness, let us sketch a proof. Take a non-negative smooth
function $K$, supported on $[-1,1]$, with integral $1$ along the real line, and periodize the rescaled function $K_t(s):=t^{-1}K(s/t)$. On $\R/\Z$, set
\[
\Phi^+=\ind_{[-\delta-t,\delta+t]}*K_t,\qquad
\Phi^-=\ind_{[-\delta+t,\delta-t]}*K_t,
\]
where the convolutions here are on $\R/\Z$.
The inequalities $0\leq \Phi^-\leq \ind_{\|\cdot\|\leq\delta}\leq \Phi^+\leq1$ and the estimation of $\widehat\Phi^{\pm}(0)$ are immediate. The Fourier transforms
of the interval indicator functions are $O(\min\{\delta,|h|^{-1}\})$, while integration by
parts gives $\widehat K_t(h)\ll_A(1+|h|t)^{-A}$. Hence, the bounds on $\hat\Phi^{\pm}(h)$ follow by multiplying the above two bounds since $\hat\Phi^{\pm}=\hat\ind_{[-\delta\mp t,\delta\pm t]}\cdot \hat K_t$. Using the resulting bound on $\hat\Phi^{\pm}(h)$ to estimate the tail proves \eqref{eq:fourierexpansion}.
\end{proof}

\begin{lemma}[Exceptional-character rough sum]\label{lem:roughlambda}
Let $\chi$ be a primitive quadratic character modulo $D\geq 2$ and set $\lambda=\ind\ast\chi$. Assume also that $L(\cdot,\chi)$ has a real zero $\beta_{\chi}$ such that $\beta_\chi=1-1/(\xi\log D)$ for some $\xi\geq 10$. If $x=D^V$ for some $V>2$, and
$$z=D^{\min\{\sqrt{V/(\log \xi)},\,2\}},$$
then there exists some absolute constant $c>0$ such that
\begin{eqnarray*}
\sum_{\substack{z<n\leq x\\P^-(n)>z}}\frac{\lambda(n)}{n}\ll\frac{V^3}{\xi}+\exp\Big(\!\!-c\sqrt{V\log \xi}\Big).
\end{eqnarray*}    
\end{lemma}
\begin{proof}
By \cite[Lemma~2.2]{MM23}, the sum is
$\ll(v^{-2}\xi^{-v/2}+V/\xi+D^{-v})(V/v)^2$, where
$v=\min\{\sqrt{V/\log\xi},2\}$. If $v=2$, this is
$O(V^3/\xi)$, since $\xi\ll D^2$ as noted after
Theorem~\ref{thm:asymp}. Otherwise $\log\xi>V/4$ and
$\log\xi\ll\log D$, so $v\log\xi=\sqrt{V\log\xi}$ and
$v\log D\gg\sqrt{V\log\xi}$. The polynomial factors are
absorbed by $\exp(-c\sqrt{V\log\xi})$.
\end{proof}

\begin{lemma}[$\beta$-sieve weights]\label{lem:sieveweights}
Let $z>1$ and $u>2$. For $r\in\N$, we define the parameters $z_r\vcentcolon=z^{((u-2)/u)^r}$. There exist two real arithmetic functions $w^+$ and $w^-$ such that
\vspace{1mm}
\begin{enumerate}[(i)]
\item $w^{\pm}(1)=1$ and $|w^{\pm}(n)|\leq 1\,\text{ for all }\,n\in\N;$
\vspace{1mm}
\item $\text{supp}(w^{\pm})\subseteq \{d\in\N:d\mid \prod_{p\leq z}p\text{ and } d\leq z^u\};$
\vspace{1mm}
\item $W^-(n)\vcentcolon=(\ind\ast w^-)(n)\leq \ind_{P^-(n)>z}\leq (\ind\ast w^+)(n)=\vcentcolon W^+(n)\,\text{ for all }\,n\in\N;$
\item $\displaystyle{|W^{\pm}(n)-\ind_{P^-(n)>z}|\ll\tau(n)^2\sum_{r\geq u/2}\ind_{P^-(n)>z_r}2^{-r}}\,
\text{ for all }\,n\in\N$.
\end{enumerate}
\end{lemma}
\begin{proof}
We follow the proof of \cite[Lemma~3.2\,(i)]{MM23}, taking $w$ to be
an upper bound beta-sieve weight $w^+$. This gives (i), (ii) and (iv),
while the upper bound in (iii) is a defining property of these weights.
The same proof, with the respective lower bound beta-sieve weight
$w^-$, gives (i), (ii) and (iv) for $w^-$ and the lower bound in (iii).
\end{proof}

\begin{lemma}[Average sieve error]\label{lem:meanerror}
Suppose that $z\geq2$, $u\geq100$, and $x\geq z^u$. Let also $f_1$ and $f_2$ be two arithmetic functions such that $|f_1(n)|,|f_2(n)|\leq \tau(n)$ for all $n\in\N$. Then, with the notation of
Lemma~\ref{lem:sieveweights}, for any signs $\sigma,\sigma'\in\{+,-\}$, we have that
\begin{gather*}
\sum_{k\ell\leq x}|f_1(k)f_2(\ell)|
\big|W^{\sigma}(k)W^{\sigma'}(\ell)-\ind_{P^-(k)>z}\ind_{P^-(\ell)>z}\big|
\ll\frac{x}{\log x}
\left(\frac{\log x}{\log z}\right)^{16}\!\!e^{-u/15}.
\end{gather*}
\end{lemma}

\begin{proof}
Put $I_z(n)=\ind_{P^-(n)>z}$ and $E^\sigma=W^\sigma-I_z$.
By Lemma~\ref{lem:sieveweights}\,(iv),
$|E^\sigma(n)|\ll\tau(n)^2\sum_{r\geq u/2}2^{-r}I_{z_r}(n)$.
Expand the difference in the lemma as
$E^\sigma(k)I_z(\ell)+I_z(k)E^{\sigma'}(\ell)+E^\sigma(k)E^{\sigma'}(\ell)$.
Writing $z_0=z$, Shiu's theorem \cite[Theorem~1]{Sh80}, applied to
the convolution of $I_{z_r}\tau^3$ and $I_{z_s}\tau^3$, gives
\[
\sum_{k\ell\leq x}I_{z_r}(k)\tau(k)^3 I_{z_s}(\ell)\tau(\ell)^3
\ll\frac{x}{\log x}
\left(\frac{\log x}{\log z_r}\right)^8
\left(\frac{\log x}{\log z_s}\right)^8.
\]
The convolution is non-negative and multiplicative, with value
$8(\ind_{p>z_r}+\ind_{p>z_s})$ at a prime $p$.
Since $\log z_r=\log z\cdot((u-2)/u)^r$, put
$B=\sum_{r\geq u/2}\{\tfrac12(u/(u-2))^8\}^r\ll e^{-u/15}$ for $u\geq100$.
The two linear errors contribute $B$, and the quadratic error contributes
$B^2$, times $x(\log x/\log z)^{16}/\log x$. This proves the lemma.
\end{proof}

\begin{lemma}[Reciprocal fractional parts]\label{lem:reciprocal}
If $\alpha>0$ is an irrational real number and $a,q$ are coprime positive integers such that $|\alpha-a/q|\leq q^{-2}$, then, for $M,X\geq1$,
\begin{equation*}
\sum_{m\leq M}\min\left\{\frac Xm,\frac1{2\|m\alpha\|}\right\}
\ll(M+q+X/q)\log(2qMX).
\end{equation*}
\end{lemma}

\begin{proof}
See \cite[Section~13.5]{IK04}.
\end{proof}

\begin{lemma}[Hyperbola estimates]\label{lem:hyperbola}
Let $\alpha>0$ be irrational and take $x\geq2$. Let $a$ and $q$ be coprime positive integers such that $|\alpha-a/q|\leq q^{-2}$. Let also
$H,Y\in[1,x]$ and assume that $b\in\Z\setminus\{0\}$ with $|b|Y\leq x$. Suppose that $\chi$ is a Dirichlet character modulo $D\leq x$, and for $j\in\{2,3\}$, put
\[
T_{j,h}:=\sum_{n_1\cdots n_j\leq Y}^{*}
\chi(n_1)e(\alpha bhn_1\cdots n_j).
\]
Here the star may impose interval restrictions whose sections in
each variable are intervals. Then, for every $\eps>0$,
\begin{equation*}
\sum_{1\leq h\leq H}\frac{|T_{j,h}|}{h}
\ll_{\eps,j} D^2x^\eps
\left(|b|Y^{1-1/j}+q+\frac{|b|Y}{q}\right).
\end{equation*}

For $j=2$, the same bound holds if we have an additional factor of
$\log(cn_2)$ in the sum defining $T_{j,h}$, provided $1\leq c\leq x$ and
$cn_2\leq x^2$ throughout the sum.
\end{lemma}

\begin{proof}
If $q>x$, the result follows from the trivial bound
$|T_{j,h}|\ll_jY(\log(2Y))^{j-1}$, so we may assume that $q\leq x$.
Partition the sum according to which of the variables $n_1,\ldots,n_j$ is the largest, and then, in each of the resulting sums, fix the remaining variables so that we are first summing over the largest one. Next, group the smaller variables by their product $r$. Each such product $r$ satisfies the bound $r\leq Y^{1-1/j}$ and has at most $\tau_{j-1}(r)$ representations, where $\tau_{j-1}$ is the $(j-1)$-fold divisor function. Split the largest variable into residue classes modulo $D$, making $\chi(n_1)$ constant when the largest variable is $n_1$. Up to a factor of modulus $1$, the inner sums have the form $\sum_{n\in I}e(\alpha bDhrn)$, where $I$ is an interval of $\R$. Hence, the geometric-sum bound and $\tau_{j-1}(r)\ll_{\eps,j}r^{\eps/4}$ give
\begin{align*}
|T_{j,h}|\!\ll_j\!D\!\!\sum_{r\leq Y^{1-1/j}}\!\!\tau_{j-1}(r)\min\left\{\frac Yr,\frac1{2\|\alpha bDhr\|}\right\}\ll_{\eps,j}\!Dx^{\eps/4}\!\!\sum_{r\leq Y^{1-1/j}}\!\!\min\left\{\frac Yr,\frac1{2\|\alpha bDhr\|}\right\}.
\end{align*}

Now, split the sum of interest $\sum_{1\leq h\leq H}|T_{j,h}|/h$ into $O(\log H)$ dyadic ranges $J<h\leq2J$. Making use of the established bound on $T_{j,h}$ in each of these ranges, we infer that
\begin{align*}
\sum_{h\sim J}\frac{|T_{j,h}|}{h}&\ll_{\eps,j}\frac{Dx^{\eps/4}}J\sum_{\substack{h\sim J\\r\leq Y^{1-1/j}}}\min\left\{\frac{2|b|DYJ}{|b|Dhr},\frac1{2\|\alpha bDhr\|}\right\}\\
&\leq\frac{Dx^{\eps/4}}J\sum_{m\leq 2|b|DJY^{1-1/j}}\tau(m)\min\left\{\frac{2|b|DYJ}m,\frac1{2\|\alpha m\|}\right\}.
\end{align*}
We apply $\tau(m)\ll_{\eps}m^{\eps/100}$ and then
Lemma~\ref{lem:reciprocal} with
$M=2|b|DJY^{1-1/j}$ and $X=2|b|DJY$.
After division by $J$, its three terms are $\ll$
$D|b|Y^{1-1/j}$, $q$, and $D|b|Y/q$, respectively.
It follows that $\sum_{h\sim J}|T_{j,h}|/h\ll_{\eps,j}D^2x^{3\eps/4}(|b|Y^{1-1/j}+q+|b|Y/q)$. Summing the $O(\log H)$ dyadic ranges, we absorb the logarithm into $x^\eps$.

For the second part of the lemma, partial summation handles the logarithmic factor; the preceding argument then applies.
\end{proof}

We shall use the following consequences of the parameter choices:
\begin{equation}\label{eq:parametererrors}
\left(\frac Vv\right)^{16}e^{-\eps V/(1500v)}
+x^{-\eps/20}\ll_\eps\calE.
\end{equation}
Indeed, $V/v\gg\sqrt{V\log\xi}$, so the first term is
absorbed by the exponential in $\calE$, after decreasing $c_\eps$.
Moreover, the quadratic class-number lower bound
$L(1,\chi)\gg D^{-1/2}$ and the standard estimate
$L(1,\chi)\ll(1-\beta_\chi)(\log D)^2$
(see \cite[Theorem~11.4]{montgomery2007multiplicative})
give $\xi\ll D^{1/2}\log D$. Since $x=D^V$, the second term
is also absorbed by the exponential in $\calE$.

\section{Proof of Proposition \texorpdfstring{\ref{prop:s1}}{3.1}:
Estimation of \texorpdfstring{$S_1$}{S1}}\label{sec:s1}

Since $\Lambda(\ell)\leq\log x$ for $\ell\leq x$ and $\lambda\geq0$, for $z<t,$ the elementary rough-number bound $\sum_{\ell\leq t, P^-(\ell)>z}1\ll t/(\log z)$ (it can readily be derived by applying Shiu's theorem \cite[Theorem~1]{Sh80} with the non-negative multiplicative function $f=\ind_{P^-(\cdot)>z}$), combined with Lemma~\ref{lem:roughlambda} and Chebyshev's bound $\psi(x)\gg x$, implies that
\begin{align*}
S_1
&\leq(\log x)\sum_{\substack{z<k<x/z\\P^-(k)>z}}
\lambda(k)\sum_{\substack{1<\ell\leq x/k\\P^-(\ell)>z}}1\ll x \cdot \frac Vv\sum_{\substack{z<k\leq x\\P^-(k)>z}}
\frac{\lambda(k)}k\notag\\
&\ll x \cdot \frac Vv\left\{\frac{V^3}{\xi}
+\exp(-c\sqrt{V\log\xi})\right\}\ll x\calE\ll \psi(x)\calE
\end{align*}
for some appropriate constant $c_{\eps}>0$. This proves Proposition~\ref{prop:s1}. 

The same argument yields the bound
\begin{equation}\label{eq:roughpairmean}
\log x\sum_{\substack{k\ell\leq x\\k,\ell>1\\P^-(k\ell)>z}}
\lambda(k)\ll_\eps \psi(x)\calE,
\end{equation}
which shall be used in the very next section.

\section{Proof of Proposition \texorpdfstring{\ref{prop:s2}}{3.2}:
Estimation of \texorpdfstring{$S_2$}{S2}}\label{sec:s2}

We may assume that $x$ is large enough in terms of $\eps$. Then, in Lemma~\ref{lem:fourier}, take
\begin{equation}\label{eq:Asmoothing}
t=\delta x^{-\eps/10},\qquad
H=\left\lceil\delta^{-1}x^{\eps/5}\right\rceil.
\end{equation}
By the non-negativity of $\lambda$, the inequality $\ind_{\|\cdot\|\leq \delta}\leq \Phi^+$, the crude bound $\Lambda(\ell)\leq \log x$ for $\ell\leq x$, and Lemma \ref{lem:sieveweights}(iii), we obtain
\[
S_2\leq(\log x)\sum_{\substack{k\ell\leq x\\k,\ell>1}}
\lambda(k)W^+(k)W^+(\ell)\Phi^+(\alpha k\ell+\beta).
\]
Put
\[
A_h:=\sum_{\substack{k\ell\leq x\\k,\ell>1}}
\lambda(k)W^+(k)W^+(\ell)e(\alpha hk\ell)\qquad (h\in\Z),
\]
and use the Fourier expansion \eqref{eq:fourierexpansion} of $\Phi^+$,
choosing $A$ sufficiently large in terms of $\eps$ so that its
pointwise error is $O_\eps(t)=O_\eps(\delta x^{-\eps/10})$. Hence,
\begin{align}\label{S2}
S_2\leq (\log x)(2\delta+O_\eps(\delta x^{-\eps/10}))A_0
+\log x\sum_{0<|h|\leq H}\widehat\Phi^+(h)e(h\beta)A_h.
\end{align}

Since $\lambda(n)\leq \tau(n)$ for all $n\in\N$, combining Lemma \ref{lem:meanerror} with \eqref{eq:roughpairmean}, we infer that
\begin{equation}\label{A0}
(\log x)A_0\ll_\eps x\Big(\calE+(V/v)^{16}e^{-\eps V/(1500v)}\Big)\!\ll_\eps x\calE
\end{equation}
for some suitable constant $c_\eps>0$. 

Inserting the bounds \eqref{bPhi} and \eqref{A0} into \eqref{S2}, it follows that
\begin{equation}\label{eq:S2Fourier}
S_2\ll_\eps x\delta\calE+
\log x\sum_{0< |h|\leq H}\bigg|\frac{A_h}{h}\bigg|.
\end{equation}

To estimate the sum involving $A_h$, we expand
$\lambda(k)=\sum_{ab=k}\chi(a)$ and the two sieve weight
convolutions in its definition. This gives
\begin{align*}
A_h=\sum_{\substack{d_1,d_2\leq z^u\\\mu(d_1)^2=1}}w^+(d_1)w^+(d_2)\sum_{\substack{ab\ell\leq x/d_2\\ab>1,\,\ell>d_2^{-1}\\d_1\mid ab}}\chi(a)e(\alpha ab\ell d_2h).
\end{align*}
Since $d_1$ is square-free, the integer $d=(a,d_1)$ is the only common positive divisor of $a$ and $d_1$ satisfying $(a/d,d_1/d)=1$. Since $d_1\mid ab$, it follows that $(d_1/d)\mid (a/d)\cdot b$, which, combined with the coprimality $(a/d,d_1/d)=1$, implies that $(d_1/d)\mid b$. Hence, by M\"obius inversion,
\begin{equation}\label{eq:square-freeidentity}
\ind_{d_1\mid ab}
=\sum_{\substack{d\mid d_1,\,d\mid a\\(d_1/d)\mid b}}
\ind_{(a/d,d_1/d)=1}=\sum_{\substack{dr\mid d_1,\,dr\mid a\\(d_1/d)\mid b}}\mu(r),
\end{equation}
where $\mu$ is the M\"obius function. We now use the identity \eqref{eq:square-freeidentity} to detect the divisibility condition $d_1\mid ab$ in the inner sums of $A_h$ above, and writing $a=drn_1,b=(d_1/d)n_2$ and $\ell=d_2n_3$, we obtain
\begin{gather*}
|A_h|\leq \sum_{\substack{d_1,d_2\leq R\\\mu^2(d_1)=1}}
\sum_{d\mid d_1}\sum_{r\mid d_1/d}
\bigg|\sum_{\substack{n_1n_2n_3\leq x/(rd_1d_2)\\
d_1rn_1n_2>1,\ d_2n_3>1}}
\chi(n_1)e(\alpha hrd_1d_2n_1n_2n_3)\bigg|.
\end{gather*}
By Lemma~\ref{lem:hyperbola} with $b=rd_1d_2$ for the sum over the positive $h$ and with $b=-rd_1d_2$ for the sum over $h<0$ (in each case $|b|\leq R^3$), we infer that
\begin{align*}
\log x\sum_{0<|h|\leq H}\frac{|A_h|}{|h|}
&\ll_\nu D^2x^\nu R
\left(x^{2/3}R+q+\frac xq\right)\sum_{d_1\leq R}\tau_3(d_1)\\
&\ll_\eps x^{2/3+\eps/4}.
\end{align*}
In the last line choose $\nu$ sufficiently small in terms of $\eps$,
and use $R=x^{\eps/100}$, $D=x^{1/V}$, and
\eqref{eq:A-diophantine-hypothesis}. Since
$x^{-\eps/20}\ll_\eps\calE$ by \eqref{eq:parametererrors},
this and \eqref{eq:S2Fourier} prove Proposition~\ref{prop:s2}.

\section{Proof of Proposition \texorpdfstring{\ref{prop:delta}}{3.3}:
Estimation of \texorpdfstring{$\Delta$}{Delta}}\label{sec:delta}

Use the same Fourier parameters as in \eqref{eq:Asmoothing}, and set
\[
M_0:=\sum_{n\leq x}\lambda'(n)\ind_{P^-(n)>z},\qquad
B_h^\pm:=\sum_{n\leq x}\lambda'(n)W^\pm(n)e(\alpha hn).
\]
Since $\lambda'(n)\leq\tau(n)\log n$, Lemma~\ref{lem:meanerror} gives
\begin{equation}\label{eq:DeltaMean}
B_0^\pm=M_0+O_\eps(x\calE).
\end{equation}
For the upper bound, use
$\ind_{P^-(n)>z}\Bohr{n}\leq W^+(n)\Phi^+(\alpha n+\beta)$.
For the lower bound, first use the non-negative minorant and then the
lower sieve weight:
\begin{equation*}
\ind_{P^-(n)>z}\Bohr{n}\geq \ind_{P^-(n)>z}\Phi^-(\alpha n+\beta)
\geq W^-(n)\Phi^-(\alpha n+\beta).
\end{equation*}
This order is important, since $W^-$ need not be non-negative.
Multiplying by $\lambda'(n)\geq0$ and summing, then using
\eqref{eq:DeltaMean} and \eqref{bPhi}, yields
\begin{equation}\label{eq:DeltaFourier}
|\Delta|\ll_\eps x\delta\calE+
\sum_{\sigma\in\{+,-\}}\sum_{1\leq h\leq H}\frac{|B_h^\sigma|}{h}.
\end{equation}
The negative frequencies give the same contribution, since the
weights are real. Moreover, $|W^\pm(n)|\leq\tau(n)$ and
$\sum_{n\leq x}\lambda'(n)|W^\pm(n)|\ll x(\log x)^C$.
Thus the Fourier mean and tail errors are
$O(tx(\log x)^C)\ll x\delta\calE$ by
\eqref{eq:parametererrors}, also for the possibly signed lower
sieve weight.

Expanding $\lambda'=\chi*\log$ and using
\eqref{eq:square-freeidentity}, we have
\begin{align*}
B_h^\sigma
=&\sum_{d\leq R}w^\sigma(d)
\sum_{e\mid d}\sum_{r\mid d/e}\mu(r)\chi(er)\\
&\sum_{n_1n_2\leq x/(dr)}
\chi(n_1)\log((d/e)n_2)e(\alpha hdrn_1n_2).
\end{align*}
There are $O(R(\log R)^2)$ outer terms, with $dr\leq R^2$.
The two-variable case of Lemma~\ref{lem:hyperbola} therefore gives
\begin{align*}
\sum_{1\leq h\leq H}\frac{|B_h^\sigma|}{h}
&\ll_\nu D^2x^\nu R
\left(x^{1/2}R+q+\frac xq\right)\\
&\ll_\eps x^{2/3+\eps/4}.
\end{align*}
Combining this with \eqref{eq:DeltaFourier} and
\eqref{eq:parametererrors} proves
Proposition~\ref{prop:delta}.

\section{Proof of Proposition \texorpdfstring{\ref{prop:typei}}{4.1}:
Type I information}\label{sec:typei}

We use the Type I argument underlying \cite[Lemma~3]{MR686496},
separating the complementary variable into residue classes modulo $D$.
Fix $s>0$ sufficiently small that
\begin{equation}\label{eq:typeimargin}
4s<\min\{\gamma-\eta,\kappa-\eta\}.
\end{equation}
In Lemma~\ref{lem:fourier}, take
\[
t=\delta y^{-s},\qquad H=\lceil\delta^{-1}y^{2s}\rceil.
\]
Since $Ht\asymp y^s$, choosing the Fourier-tail exponent sufficiently
large makes the error in \eqref{eq:fourierexpansion} $O(t)$.

Let $\varpi(m)=1$ or $\chi(m)$. For each $d\leq M$ and residue class
$r\pmod D$, put
\[
J_{d,r}:=\{m\in J_d:m\equiv r\pmod D\},\qquad
S_{d,r}:=\sum_{m\in J_{d,r}}(\Bohr{dm}-2\delta).
\]
Summing the inequalities $\Phi^-\leq\ind_{\|\cdot\|\leq\delta}
\leq\Phi^+$ over $J_{d,r}$, \eqref{eq:fourierexpansion} gives
\begin{equation*}
|S_{d,r}|\ll t|J_{d,r}|+
\delta\sum_{0<|h|\leq H}
\left|\sum_{m\in J_{d,r}}e(\alpha hdm)\right|.
\end{equation*}
Here $e(h\beta)$ has absolute value $1$. Since $\varpi(m)$ and
$\ind_{(m,D)=1}$ are constant on $J_{d,r}$, the two sums in the
proposition are bounded in absolute value by
\begin{equation}\label{eq:typeimaster}
\ll t y(\log y)^C
+\delta\sum_{0<|h|\leq H}\sum_{d\leq M}|c_d|
\sum_{r\!\Mod{\!D}}
\left|\sum_{m\in J_{d,r}}e(\alpha hdm)\right|.
\end{equation}
The first term follows from
$\sum_{d\leq M}|c_d|/d\ll(\log y)^C$.

For $|h|\sim K$ and $d\sim N$, the geometric-sum bound, followed by
Lemma~\ref{lem:reciprocal} with $\ell=D|h|d$, gives, for every $\nu>0$,
\begin{align*}
&\sum_{|h|\sim K}\sum_{d\sim N}|c_d|
\sum_{r\!\Mod{\!D}}
\left|\sum_{m\in J_{d,r}}e(\alpha hdm)\right|\\
&\qquad\ll_\nu D y^\nu
\sum_{|h|\sim K}\sum_{d\sim N}
\min\left\{\frac{y}{Dd},\frac1{2\|\alpha Dhd\|}\right\}\\
&\qquad\ll_\nu D^2y^{2\nu}
\left(KN+q+\frac{yK}{q}\right).
\end{align*}
Indeed, $DM=o(y)$, so each progression has
$O(y/(Dd))$ terms; the multiplicity of $\ell$ is divisor bounded.
Summing the dyadic blocks in \eqref{eq:typeimaster}, we obtain
\begin{equation*}
\ll_\nu D^2y^{3\nu}
\left\{\delta\left(HM+q+\frac{yH}{q}\right)+ty\right\}.
\end{equation*}

Since $x/\log x\leq y\leq x/2$, $D=y^{o(1)}$, and
$X=\delta y^{1-o(1)}$, the four terms in braces, divided by
$\delta y$, are at most
\[
x^{\eta-\gamma+2s+o(1)},\qquad x^{-\kappa+o(1)},
\qquad x^{\eta-\kappa+2s+o(1)},\qquad y^{-s},
\]
respectively. By \eqref{eq:typeimargin}, all are power savings;
choosing $\nu$ sufficiently small proves both assertions.

\section{Proof of Proposition \texorpdfstring{\ref{prop:sievehypothesis}}{4.2}:
Sieve hypotheses}\label{sec:sievehypothesis}

We verify the reference-sequence and local-density conditions of
\cite[Theorem~16]{Mer24}, as well as the exceptional-character bounds.
Throughout,
$D=y^{o(1)}$ and $D\gg_A(\log y)^A$ for every fixed $A>0$.

First, M\"obius inversion gives, for every interval $I$,
\[
\sum_{\substack{m\in I\\(m,D)=1}}1
=\frac{\varphi(D)}D \cdot |I|+O(\tau(D)).
\]
Consequently, for $d\leq y^{3/4}$,
\begin{equation*}
\sum_m b_{dm}
=\frac{2\delta y\varphi(D)}{Dd}\cdot \ind_{(d,D)=1}
+O(\delta\tau(D)\ind_{(d,D)=1}).
\end{equation*}
Consequently, for divisor-bounded $(c_d)$ and $M\leq y^{3/4}$,
\[
\sum_{d\leq M}c_d\sum_m b_{dm}
=X\sum_{d\leq M}c_dg(d)
+O_A(X(\log y)^{-A}).
\]
If $I=(N,N(1+\Delta)]$ with $\Delta\geq(\log y)^{-B}$,
$dN\asymp y$, and $d\leq y^{3/4}$, then the same counting gives
$\sum_{m\in I}b_{dm}\ll\Delta Xg(d)$, uniformly for fixed $B$.
Indeed, $\Delta N\gg y^{1/4}(\log y)^{-B}$, whereas
$\tau(D)=y^{o(1)}$ and $\varphi(D)/D\gg1/\log\log D$;
thus the counting error is absorbed by the main bound.
This is the short-interval bound for the reference sequence;
Proposition~\ref{prop:typei} supplies the corresponding comparison
between the two sequences required in \cite[Proposition~8]{Mer24}.

By periodicity, $\sum_{m\in I}\chi(m)=O(D)$ for every interval $I$.
Hence, for $M\leq y^{3/4}$ and arbitrary $J_d$,
\begin{equation*}
\sum_{d\leq M}c_d\sum_{m\in J_d}\chi(m)b_{dm}
\ll\delta DM(\log y)^C
\ll_AX(\log y)^{-A}.
\end{equation*}
Together with \eqref{eq:typeitwist}, this verifies
\cite[Proposition~9]{Mer24}. Taking, for example, a fixed
$0<\eps'<1/12$ covers the reference level $y^{2/3+\eps'}$ used there.

By \eqref{eq:g}, $g$ is multiplicative, $0\leq g(p)<1$,
and $g(d)\leq1/d$. For $2\leq w<z$, Mertens' theorem gives
$\prod_{w\leq p<z}(1-g(p))^{-1}
\leq (1+O(1/\log w))\log z/\log w$.
Thus the local density satisfies the condition in
\cite[Lemma~5]{Mer24}. For $w>D$,
\begin{equation}\label{eq:mertensD}
\prod_{p\leq w}(1-g(p))
=(1+o(1)) \cdot \frac{\ee^{-\gamma_*}}{\log w} \cdot \frac D{\varphi(D)}.
\end{equation}
Also, the prime number theorem gives $\sum_n b_n\Lambda(n)\sim2\delta y$:
the terms with prime base dividing $D$ have total weight
$O(\omega(D)\log y)=y^{o(1)}$. Together with \eqref{eq:X}
and \eqref{eq:mertensD}, this proves \eqref{eq:normalization}.
For each fixed $a>0$ and $w>y^a$, the contribution of $p\mid D$
to $\sum_{k\sim w}\Lambda(k)/k$ is
$O(\omega(D)\log w/w)=o(1)$. Thus the same theorem gives
\eqref{eq:referenceg}.

For the exceptional-character condition, note that
$\chi(n)g(n)=\chi(n)/n$. Periodicity and partial summation give
\begin{equation}\label{eq:chitail}
\sum_{n\leq t}\frac{\chi(n)}n=L(1,\chi)+O(D/t).
\end{equation}
The Dirichlet hyperbola method also gives
\begin{equation}\label{eq:lambdasum}
\sum_{n\leq t}\lambda(n)=L(1,\chi)t+O(D\sqrt t).
\end{equation}
Indeed, put $H=\lfloor\sqrt t\rfloor$ and
$S_\chi(v)=\sum_{n\leq v}\chi(n)=O(D)$. Splitting $ab\leq t$
according to $a\leq H$ or $b\leq H$ gives
\[
\sum_{n\leq t}\lambda(n)
=\sum_{a\leq H}\chi(a)\left\lfloor\frac ta\right\rfloor
+\sum_{b\leq H}S_\chi(t/b)-H S_\chi(H)
=t\sum_{a\leq H}\frac{\chi(a)}a+O(D\sqrt t).
\]
The tail of the truncated sum is $O(D/H)$ by partial summation,
giving \eqref{eq:lambdasum}. The quadratic class-number formula gives
$L(1,\chi)\gg D^{-1/2}$,
so the errors in
\eqref{eq:chitail} and \eqref{eq:lambdasum} are absorbed by their
main bounds for $t\geq D^9$. Partial summation of
\eqref{eq:lambdasum} gives
$\sum_{w<n\leq t}\lambda(n)/n\ll L(1,\chi)\log t$
for $t>w\geq D^9$, since $D/\sqrt w\ll L(1,\chi)$.
Since $0\leq\lambda(n)g(n)\leq\lambda(n)/n$, this proves
\eqref{eq:exceptionalaxiom}. This is stronger than the logarithmic
bound required in \cite[Proposition~10]{Mer24}.

Finally, fix $0<a<1$. If $n\sim y$, $\Lambda(n)\ne0$, and
$(n,P(y^a))>1$, then $n$ is a proper prime power. Therefore
\[
\sum_{\substack{n\sim y\\(n,P(y^a))>1}}
(a_n+b_n)\Lambda(n)
\ll y^{1/2}(\log y)^2=o(\delta y),
\]
since $1/3+\eta<1/2$ and $y=x^{1+o(1)}$.
This verifies the remaining prime-power condition and completes the
proof of Proposition~\ref{prop:sievehypothesis}.

\section*{Acknowledgements}
During the preparation of this work, S.-K. Leung was supported by the
Croucher Fellowship for Postdoctoral Research. S. Sachpazis acknowledges
support from the grants PRIMUS/25/SCI/008 and PRIMUS/25/SCI/0017 of
Charles University.

%\nocite{*}
%\raggedbottom
\printbibliography

@article {MR2525926,
    AUTHOR = {Matom\"aki, K.},
     TITLE = {The distribution of {$\alpha p$} modulo one},
   JOURNAL = {Math. Proc. Cambridge Philos. Soc.},
  FJOURNAL = {Mathematical Proceedings of the Cambridge Philosophical
              Society},
    VOLUME = {147},
      YEAR = {2009},
    NUMBER = {2},
     PAGES = {267--283},
  %
  %ISSN = {0305-0041,1469-8064},
 %  MRCLASS = {11K06 (11K31 11L07 11N36)},
%  MRNUMBER = {2525926},
%MRREVIEWER = {R.\ C.\ Baker},
 %      DOI = {10.1017/S030500410900245X},
%       URL = {https://doi.org/10.1017/S030500410900245X},
}

@book {Har07,
    AUTHOR = {Harman, G.},
     TITLE = {Prime-detecting sieves},
    SERIES = {London Mathematical Society Monographs Series},
    VOLUME = {33},
 PUBLISHER = {Princeton University Press, Princeton, NJ},
      YEAR = {2007},
     PAGES = {xvi+362},
   %   ISBN = {978-0-691-12437-7},
  % MRCLASS = {11N36 (11N25 11N35)},
 % MRNUMBER = {2331072},
%MRREVIEWER = {S.\ W.\ Graham},
}

@article {MR686496,
    AUTHOR = {Harman, G.},
     TITLE = {On the distribution of {$\alpha p$}\ modulo one},
   JOURNAL = {J. London Math. Soc. (2)},
  FJOURNAL = {Journal of the London Mathematical Society. Second Series},
    VOLUME = {27},
      YEAR = {1983},
    NUMBER = {1},
     PAGES = {9--18},
   %   ISSN = {0024-6107,1469-7750},
 %  MRCLASS = {10F40 (10H32)},
%  MRNUMBER = {686496},
%MRREVIEWER = {Matti\ Jutila},
  %     DOI = {10.1112/jlms/s2-27.1.9},
 %      URL = {https://doi.org/10.1112/jlms/s2-27.1.9},
}

@article {MR1367078,
    AUTHOR = {Harman, G.},
     TITLE = {On the distribution of {$\alpha p$} modulo one. {II}},
   JOURNAL = {Proc. London Math. Soc. (3)},
  FJOURNAL = {Proceedings of the London Mathematical Society. Third Series},
    VOLUME = {72},
      YEAR = {1996},
    NUMBER = {2},
     PAGES = {241--260},
   %   ISSN = {0024-6115,1460-244X},
  % MRCLASS = {11J71 (11N35 11N36)},
 % MRNUMBER = {1367078},
%MRREVIEWER = {D.\ R.\ Heath-Brown},
 %      DOI = {10.1112/plms/s3-72.2.241},
%       URL = {https://doi.org/10.1112/plms/s3-72.2.241},
}

@article {MR1863396,
    AUTHOR = {Heath-Brown, D. R. and Jia, C.},
     TITLE = {The distribution of {$\alpha p$} modulo one},
   JOURNAL = {Proc. London Math. Soc. (3)},
  FJOURNAL = {Proceedings of the London Mathematical Society. Third Series},
    VOLUME = {84},
      YEAR = {2002},
    NUMBER = {1},
     PAGES = {79--104},
   %   ISSN = {0024-6115,1460-244X},
  % MRCLASS = {11N36 (11J71)},
 % MRNUMBER = {1863396},
%MRREVIEWER = {R.\ C.\ Baker},
   %    DOI = {10.1112/plms/84.1.79},
  %     URL = {https://doi.org/10.1112/plms/84.1.79},
}

@article {MR1790174,
    AUTHOR = {Jia, C.},
     TITLE = {On the distribution of {$\alpha p$} modulo one. {II}},
   JOURNAL = {Sci. China Ser. A},
  FJOURNAL = {Science in China. Series A. Mathematics},
    VOLUME = {43},
      YEAR = {2000},
    NUMBER = {7},
     PAGES = {703--721},
   %   ISSN = {1006-9283,1862-2763},
  % MRCLASS = {11J71 (11K60 11N36)},
 % MRNUMBER = {1790174},
%MRREVIEWER = {John\ H.\ Loxton},
  %     DOI = {10.1007/BF02878436},
 %      URL = {https://doi.org/10.1007/BF02878436},
}

@article {MR1247382,
    AUTHOR = {Jia, C.},
     TITLE = {On the distribution of {$\alpha p$} modulo one},
   JOURNAL = {J. Number Theory},
  FJOURNAL = {Journal of Number Theory},
    VOLUME = {45},
      YEAR = {1993},
    NUMBER = {3},
     PAGES = {241--253},
   %   ISSN = {0022-314X,1096-1658},
  % MRCLASS = {11K60 (11K06 11N36)},
 % MRNUMBER = {1247382},
%MRREVIEWER = {John\ H.\ Loxton},
 %      DOI = {10.1006/jnth.1993.1075},
%       URL = {https://doi.org/10.1006/jnth.1993.1075},
}

@article {MM23,
    AUTHOR = {Matom\"aki, K. and Merikoski, J.},
     TITLE = {Siegel zeros, twin primes, {G}oldbach's conjecture, and primes
              in short intervals},
   JOURNAL = {Int. Math. Res. Not. IMRN},
  FJOURNAL = {International Mathematics Research Notices. IMRN},
      YEAR = {2023},
    NUMBER = {23},
     PAGES = {20337--20384},
   %   ISSN = {1073-7928,1687-0247},
  % MRCLASS = {11M20 (11N05)},
 % MRNUMBER = {4675073},
%MRREVIEWER = {Todd\ Molnar},
   %    DOI = {10.1093/imrn/rnad069},
  %     URL = {https://doi.org/10.1093/imrn/rnad069},
}

@book{montgomery2007multiplicative,
  title={Multiplicative number theory I: Classical theory},
  author={Montgomery, H. L. and Vaughan, R. C.},
  volume={97},
  year={2007},
  publisher={Cambridge University Press}
}

@article {MR472731,
    AUTHOR = {Vaughan, R. C.},
     TITLE = {On the distribution of {$\alpha p$} modulo {$1$}},
   JOURNAL = {Mathematika},
  FJOURNAL = {Mathematika. A Journal of Pure and Applied Mathematics},
    VOLUME = {24},
      YEAR = {1977},
    NUMBER = {2},
     PAGES = {135--141},
   %   ISSN = {0025-5793},
  % MRCLASS = {10H15 (10G05)},
 % MRNUMBER = {472731},
%MRREVIEWER = {Matti\ Jutila},
   %    DOI = {10.1112/S0025579300009025},
  %     URL = {https://doi.org/10.1112/S0025579300009025},
}

@book {MR2104806,
    AUTHOR = {Vinogradov, I. M.},
     TITLE = {The method of trigonometrical sums in the theory of numbers},
      NOTE = {Translated from the Russian, revised and annotated by K. F.
              Roth and Anne Davenport,
              Reprint of the 1954 translation},
 PUBLISHER = {Dover Publications, Inc., Mineola, NY},
      YEAR = {2004},
     PAGES = {x+180},
  %    ISBN = {0-486-43878-3},
 %  MRCLASS = {11Lxx (01A75)},
%  MRNUMBER = {2104806},
}

@article {MR0703977,
    AUTHOR = {Heath-Brown, D. R.},
     TITLE = {Prime twins and {S}iegel zeros},
   JOURNAL = {Proc. London Math. Soc. (3)},
  FJOURNAL = {Proceedings of the London Mathematical Society. Third Series},
    VOLUME = {47},
      YEAR = {1983},
    NUMBER = {2},
     PAGES = {193--224},
      ISSN = {0024-6115,1460-244X},
 %  MRCLASS = {10H08 (10H32 10J15)},
%  MRNUMBER = {703977},
%MRREVIEWER = {Matti\ Jutila},
 %      DOI = {10.1112/plms/s3-47.2.193},
 %      URL = {https://doi.org/10.1112/plms/s3-47.2.193},
}

@article {Mer24,
    AUTHOR = {Merikoski, J.},
     TITLE = {Exceptional characters and prime numbers in sparse sets},
   JOURNAL = {Algebra Number Theory},
  FJOURNAL = {Algebra \& Number Theory},
    VOLUME = {18},
      YEAR = {2024},
    NUMBER = {7},
     PAGES = {1305--1332},
  %    ISSN = {1937-0652,1944-7833},
%   MRCLASS = {11N32 (11N36)},
%  MRNUMBER = {4757307},
%MRREVIEWER = {Karin\ Halupczok},
 %      DOI = {10.2140/ant.2024.18.1305},
 %      URL = {https://doi.org/10.2140/ant.2024.18.1305},
}

@book {IK04,
    AUTHOR = {Iwaniec, H. and Kowalski, E.},
     TITLE = {Analytic number theory},
    SERIES = {American Mathematical Society Colloquium Publications},
    VOLUME = {53},
 PUBLISHER = {American Mathematical Society, Providence, RI},
      YEAR = {2004},
     PAGES = {xii+615},
  %    ISBN = {0-8218-3633-1},
%   MRCLASS = {11-02 (11Fxx 11Lxx 11Mxx 11Nxx)},
 % MRNUMBER = {2061214},
%MRREVIEWER = {K.\ Soundararajan},
%       DOI = {10.1090/coll/053},
%       URL = {https://doi.org/10.1090/coll/053},
}

@article {Sh80,
    AUTHOR = {Shiu, P.},
     TITLE = {A {B}run-{T}itchmarsh theorem for multiplicative functions},
   JOURNAL = {J. Reine Angew. Math.},
  FJOURNAL = {Journal f\"ur die Reine und Angewandte Mathematik. [Crelle's
              Journal]},
    VOLUME = {313},
      YEAR = {1980},
     PAGES = {161--170},
   %   ISSN = {0075-4102,1435-5345},
 %  MRCLASS = {10H25},
%  MRNUMBER = {552470},
%MRREVIEWER = {A.\ I.\ Vinogradov},
 %      DOI = {10.1515/crll.1980.313.161},
 %      URL = {https://doi.org/10.1515/crll.1980.313.161},
}

\end{document}